\documentclass{article}
\usepackage{tikz-cd}
\usepackage{amsthm}
\usepackage{amsbsy}
\usepackage{amssymb}
\usepackage{amsfonts}
\usepackage{amsmath}
\usepackage{mathrsfs}
\usepackage[dvips]{lscape}
\usepackage{verbatim} 
\usepackage{enumerate}
\usepackage{comment}

\theoremstyle{definition}
\newtheorem{thm}{Theorem}[section]
\newtheorem{lem}[thm]{Lemma}
\newtheorem{prp}[thm]{Proposition}
\newtheorem{dfn}[thm]{Definition}
\newtheorem{cor}[thm]{Corollary}

\newtheorem{rmk}[thm]{Remark}

\newtheorem{exa}[thm]{Example}

\newcommand{\beq}{\begin{equation*}}
\newcommand{\eeq}{\end{equation*}}
\newcommand{\beqr}{\begin{eqnarray*}}
\newcommand{\eeqr}{\end{eqnarray*}}
\newcommand{\bal}{\begin{align*}} 
\newcommand{\eal}{\end{align*}}
\newcommand{\bei}{\begin{itemize}}
\newcommand{\eei}{\end{itemize}}

\newcommand{\dlim}{\underrightarrow\lim\,}

\newcommand{\af}{\alpha}
\newcommand{\bt}{\beta}
\newcommand{\gm}{\gamma}

\newcommand{\ld}{\lambda}

\newcommand{\om}{\omega}

\newcommand{\Gm}{\Gamma}

\newcommand{\Z}{{\mathbb{Z}}}

\newcommand{\C}{{\mathbb{C}}}
\newcommand{\T}{{\mathbb{T
}}}

\newcommand{\id}{{\mathrm{id}}}

\newcommand{\diag}{{\mathrm{diag}}}

\newcommand{\Aut}{{\mathrm{Aut}}}
\newcommand{\Ad}{{\mathrm{Ad}}}

\newcommand{\QED}{\rule{0.4em}{2ex}}

\newcommand{\cra}{\curvearrowright}

\begin{document}
\title{Classification of certain inductive limit actions of compact groups on AF algebras}
\author{Qingyun Wang}
\maketitle

\begin{abstract}
Let $A=\dlim{A_n}$ be an AF algebra, $G$ be a compact group. We consider inductive limit actions of the form $\af=\dlim{\af_n}$, where $\af_n\colon G\cra A_n$ is an action on the finite dimensional C*-algebra $A_n$ which fixes each matrix summand. If each $\af_n$ is inner, such actions are classified by equivariant K-theory in \cite{Handelman1985}. However, if the actions $\af_n$ are not inner, we show that such actions are not classifiable by equivariant K-theory. We give a complete classification of such actions using twisted equivariant K-theory.
\end{abstract}

\section{Introduction}
Throughout the whole paper, we shall restrict our attention to compact group actions, although some of the definitions and results are valid for more general groups.\\
Following the convention in \cite{Phillips1987}, if $\af\colon G\cra A$ is an action, we shall also use $(G, A, \af)$ to denote this action and call $A$ a $G$-C*-algebra. All actions are assumed to be continuous in the point-norm topology, i.e. for each $a\in A$, the map $g\rightarrow \af_g(a)$ is continuous.
\begin{dfn}
Let $A = \dlim{A_n}$ be a C*-algebra, $G$ be a compact group. Let $\af_n\colon G\cra A_n$ be a sequence of actions such that the connecting maps $\phi_{n, m}\colon A_n\rightarrow A_m$ are all equivariant. Then there is an induced action $\af\colon G\cra A$. We call $\af$ an \emph{inductive limit action} in this case. We shall write
\beq
(G, A, \af) = \dlim(G, A_n ,\af_n).
\eeq
We say $\af$ is \emph{locally representable} if each $\af_n$ is induced by a representation $G\rightarrow U(A_n)$. We say $\af$ is \emph{locally spectrally trivial} if each $\af_n$ induces trivial action on the spectrum of $A_n$.
\end{dfn}

We shall say that $\af$ is an inductive limit action on an $AF$ algebra if each $A_n$ is finite dimensional in the above definition.\\
The classification of C*-algebras admitting an inductive limit structure has been extensively studied. However, the equivariant version is far less well-understood. For example, it's not known how to classify all inductive limit actions of a finite group on a general AF algebra, except when the group is finite cyclic;
 see \cite{Elliott1996b} and \cite{Li2013}. There is a number of results on classification of locally representable inductive limit actions, see e.g. \cite{Handelman1985}, \cite{Kishimoto1990}, \cite{Bratteli1991} and \cite{Dean2009}. On the other direction, inductive limits of actions on certain homogeneous algebras, which acts trivially on the fiber but non-trivially on the spectrum, has been classified in \cite{Su1996}.

 The major invariant used in aforementioned results is the equivariant K-theory. For the class of actions we shall consider, we will show that equivariant K-theory is not enough for classification. We shall introduce the twisted equivariant K-theory for compact group actions on C*-algebras, which generalize the ordinary equivariant K-theory. The main result of this paper is a complete classification using this new invariant. We expect that this new invariant will be useful in future classifications of group actions on C*-algebras.

\section{Actions on full matrix algebras}
We start our analysis from actions on matrix algebras. we shall soon see that one has to go beyond representable actions and look for more invariants than just equivariant K-theory. For the definition of equivariant K-theory and it's properties, our standard reference is \cite{Phillips1987}.\\
Let $M_n$ be the $n\times n$ complex matrix algebra. Let $\af\colon G\cra M_n$ be an action. It’s easy to see that $\Aut(M_n)$ is isomorphic to $PU_n$ as groups, where $PU_n = U_n/\T$ is the complex projective unitary group. One can check that they are actually isomorphic as topological groups, where $PU_n$ has the standard quotient topology. Note that for $n>1$, there is no continuous section from $PU_n$ to $U_n$, so a continuous map $G\rightarrow PU_n$ does not necessarily lift to a continuous map $G\rightarrow U_n$. However, by a result of Dixmier, there is always a Borel section from $PU_n$ to $U_n$, so any continuous map $G\rightarrow PU_n$ admits a lift $G\rightarrow U_n$ which is Borel.\\
The above argument shows that there is a Borel family of unitaries $U_g$ such that $\af_g = \Ad\,U_g$ for each $g\in G$. Now use the identities $\af_g\af_h = \af_{gh}$ and $(\af_g\af_h )\af_k = \af_g(\af_h\af_k)$, we can show that there is a Borel function $\ld : G \times G \rightarrow \T$ such that:
\begin{enumerate}[(1)]
\item $U_gU_h = \ld_{g,h}U_{gh}$,
\item $\ld_{g,1}=\ld_{1,g}=1,\quad\quad\ld_{g,h} \ld_{gh,k} = \ld_{g,hk} \ld_{h,k}$.
\end{enumerate}
\begin{dfn}
A Borel function $\ld\colon G\times G \rightarrow \T$ satisfying the second condition above is called a \emph{2-cocycle}. A Borel map $\pi: G \rightarrow U_n$ satisfying $U_gU_h = \ld_{g,h} U_{gh}$ for some 2-cocycle $\ld$ is called a \emph{$\ld$-representation}, or a \emph{cocycle representation} (also called \emph{projective representation} in some literature) if we do not want to stress $\ld$. If $\ld$ is trivial (i.e. $\ld_{g,h} = 1$ for all $g, h$), we will simply say $\pi$ is a representation.
\end{dfn}
\begin{rmk}
What we have just defined should be ``unitary'' $\ld$-representation. Since we will not talk about non-unitary ones, we shall make this convention for the rest of the paper. Most concepts for representations carry over to cocycle representations naturally, like direct sum, tensor product, unitary equivalence etc, we shall use these terminologies freely. We refer to \cite{Karpilovsky1985projective} for basic facts of cocycle representations.
\end{rmk}
Now suppose $\{V_g\}$ is another Borel family of unitaries inducing $\af$, let $\tilde{\ld}$ be the corresponding 2-cocycle. Then there exist a Borel function $\mu\colon G \rightarrow \T$ such that 
\beq
U_g = \mu_g V_g \quad\quad \ld_{g,h} =\frac{\mu_g\mu_h}{\mu_{gh}}\tilde{\ld}_{g,h}
\eeq
A Borel function $s : G \times G \rightarrow \T$ of the form $\mu_g \mu_h /\mu_{gh}$ for some $\mu : G \rightarrow \T$ is called a \emph{2-coboundary}. In this case, we shall write
\beq
s=d\,\mu.
\eeq
Two cocycles differ by a coboundary is said to be \emph{cohomologous}. Let $Z^2(G, \T)$ be the set of 2-cocycles and $B^2(G, \T)$ be the set of 2-coboundaries. It’s easy to see that $B^2(G, \T)$ becomes a subgroup of $Z^2(G, \T)$ under pointwise multiplication. If $\ld\in Z^2(G, \T)$, we shall use either $\bar{\ld}$ or $\ld^{-1}$ to denote the inverse of $\ld$, where $\bar{\ld}$ is the complex conjugate of $\ld$.
\begin{dfn}(\cite{Moore1976}) The second (measurable) cohomology group of $G$, denoted by $H_m^2(G, \T)$, is the quotient group $Z^2(G, \T)/B^2(G, \T)$. 
\end{dfn}
By the above analysis, for any action $\af\colon G\cra M_n$ , there is a well-defined element $[\ld_\af]\in H_m^2(G, \T)$. We can see that $[\ld_\af] = [\ld_\bt]$ if $\af$ and $\bt$ are conjugate, and $\af$ is inner (induced by a representation) if and only if $[\ld_\af]$ is the identity element. On the other hand, any $\ld$-representation induces an action. The induced action is continuous because a Borel homomorphism between two Polish groups is automatically continuous, by a classical theorem of Banach.\\
In general, $H_m^2(G, \T)$ could be non-trivial, even for finite abelian groups $G$. It is known that for any $\ld \in Z^2(G, \T)$, there is at least one $\ld$-representation. So if we pick some $\ld$-representation such that $[\ld]$ is not the identity element of $H_m^2(G, \T)$, then it will induce an action which is pointwisely inner but not inner. Using cocycle representations, we can actually produce examples of two actions with the same equivariant K-theory, but are not conjugate.
\begin{exa}\label{fdce}
Let $G = \Z/3\Z \oplus \Z/3\Z$. Denote the generators of $G$ by $e_1$ and $e_2$. Let $\om$ be a primitive 3rd root of unity. Set\\
\beq
A_1=\left(\begin{array}{ccc}
1 & 0 & 0\\
0 & \om & 0\\
0 & 0 & \om^2
\end{array}\right),
\quad
A_2=\left(\begin{array}{ccc}
1 & 0 & 0\\
0 & \om^2 & 0\\
0 & 0 & \om
\end{array}\right),
\quad
B=\left(\begin{array}{ccc}
0 & 1 & 0\\
0 & 0 & 1\\
1 & 0 &0
\end{array}\right).
\eeq

For $i = 1, 2$, let $\pi_i\colon G \rightarrow U_3$ be defined by
\beq
\pi_i(e_1^r e_2^s) = A_i^r B^s,\quad\text{for\,\,} 0 \leq r, s \leq 2.
\eeq
Straightforward computation shows that $\pi_i$ is a cocycle representation corresponding some 2-cocycle $\ld_i$ . Denote the trivial cocycle by 1, then one can check that $[1], [\ld_1]$ and $[\ld_2]$ are all different in $H_m^2(G, \T)$ (Actually $H_m^2(G, \T) = \Z/3\Z$).\\
For $i=1, 2$, let $\af_i : G\cra M_3$ be the action induced by $\pi_i$. Then $\af_1$ and $\af_2$ are not conjugate. But their equivariant K-theories $K_0^G(A, \af_i)$ are isomorphic. To see this, we first check that both crossed products $A \rtimes_{\af_i} G$ are isomorphic to $M_9$. This could be proved either by elementary methods or referring to Example \ref{z3z3} in the next section. By Julg's theorem, $K_0^G(A, \af_1) \cong K_0^G(A, \af_2)$ as groups. For abelian group actions, one can identify the representation ring $R(G)$ with $\Z(\hat{G})$, and the module structure on the equivariant $K_0$-groups is induced by the dual action $\hat{G}\cra A\rtimes_{\af_i} G$. Since both crossed products are isomorphic to a matrix algebra, the dual group $\hat{G}$ acts trivially on $K_0(A\rtimes_{\af_i} G)$. Therefore  $K_0^G(A, \af_1) \cong K_0^G(A, \af_2)$ as modules over the representation ring.
\end{exa}
A counter-example for actions on infinite dimensional C*-algebras could also be obtained, but the construction is more sophisticated. 
\begin{exa}
We still let $G = \Z/3\Z \oplus \Z/3\Z$ and let $\pi_i , \ld_i, \af_i$ be constructed as in Example \ref{fdce}. Let $A$ be the the UHF-algebra $M_{3^\infty}$. As a consequence of \cite{Wang2015}, one can find two product-type actions (infinite tensors of inner actions on $M_3$ ) $\bt_1 , \bt_2\colon G\cra A$ such that $\bt_1$ has the tracial Rokhlin property but not the Rokhlin property, while $\bt_2$ has the Rokhlin property. Let $\gm_i = \af_i \otimes \bt_i$ be the tensor product action, for $i = 1, 2$. Then $K_0^G(A, \gm_1)\cong K_0^G(A, \gm_2)$ but $\gm_1$ and $\gm_2$ are not conjugate.
\end{exa}
Viewing $\gm_i$ as an inductive limit action on $M_3 \otimes A \cong A$ in the natural way, we can see that at each finite stage, $\gm_i$ is induced by some $\ld_i$-representation. Hence the crossed products are both isomorphic to some full matrix algebra. Arguing as in Example \ref{fdce}, and using the fact that equivariant K-theory is sequentially continuous:
\beq
K_0^G (A, \af) = \dlim K_0^G (A_n , \af_n)
\eeq
One can show that $K_0^G (A, \gm_1 ) \cong K_0^G(A, \gm_2 )$. That $\gm_1$ and $\gm_2$ are not conjugate is an immediate consequence of the following fact:
\begin{prp} Let $\af\colon G\cra A$ be an action of a finite group $G$ on a simple C*-algebra $A$. Let $\bt \colon G\cra M_n$ be an arbitrary action. Then $\af$ has the tracial Rokhlin property if and only if $\af \otimes \bt$ has the tracial Rokhlin property. $\af$ has the Rokhlin property if and only if $\af \otimes \bt$ has the Rokhlin property.
\end{prp}
\begin{proof} 
We shall only deal with the tracial Rokhlin property part. The proof of the Rokhlin property part is similar. One direction is easy: if $\af$ has the tracial Rokhlin property then so does $\af\otimes\bt$. The converse is true if $\bt$ is an inner action, by Lemma 3.9 of \cite{Phillips2011}. For arbitrary $\bt$, suppose it is induced by some $\ld$-representation $\pi$. Let $\bar{\bt}$ be the action induced by $\bar{\pi}$, the contragredient of $\pi$. We see that $\bar{\pi}$ is a $\bar{\ld}$-representation, therefore $\pi \otimes \bar{\pi}$ is a genuine representation. Now assume $\af\otimes \bt$ has the tracial Rokhlin property, then so does $\af \otimes \bt \otimes \bar{\bt}$. But now $\bt \otimes \bar{\bt}$ is an inner action, by Lemma 3.9 of \cite{Phillips2011}, $\af$ has the tracial Rokhlin property.
\end{proof}

\section{Twisted equivariant K-theory}
The example in previous section suggest us to look for additional invariant in order to obtain classification result. Before we describe the new invariant, we need some basic facts about cocycle representations. 
\begin{dfn} Let $G$ be a compact group, let $\ld$ be a 2-cocycle. We let $V^\ld (G)$ be the set of equivalent classes of $\ld$-representations of $G$. It becomes an abelian semigroup under direct sum of representations. We use $R^\ld (G)$ to denote the Grothendieck completion of $V^\ld (G)$, i.e. $R^\ld (G)$ is the group of formal differences of equivalent classes of $\ld$-representations.
\end{dfn}
We shall call $R^\ld (G)$ the \emph{$\ld$-representation group}. We also use the term \emph{cocycle representation group} when the 2-cocycle $\ld$ is not specified. $R^\ld (G)$ becomes an ordered abelian group with $V^\ld (G)$ being the positive cone. When $\ld$ is trivial, i.e. $\ld(g, h) = 1$ for any $g, h \in G$, we simply write $R^\ld (G)$ by $R(G)$, which is just the representation ring of $G$.\\

Unlike ordinary representations, when $\ld$ is nontrivial, the group $R^\ld (G)$ does not form a ring under tensor product. This is because if $\rho_i$ is a $\ld_i$ -representation, for $i = 1, 2$, then $\rho_1 \otimes \rho_2$ is a $\ld_1 \ld_2$ -representation. Instead, tensor product gives
us a bunch of pairings
\beq
R^{\ld_1} (G) \times R^{\ld_2} (G) \rightarrow R^{\ld_1\ld_2} (G),
\eeq
which are commutative and associative in the sense that
\begin{enumerate}[(1)]
\item $[\rho_1 ] \otimes [\rho_2 ] = [\rho_2 ] \otimes [\rho_1 ]$.
\item $([\rho_1 ] \otimes [\rho_2 ]) \otimes [\rho_3 ] = [\rho_1 ] \otimes ([\rho_2 ]\otimes[\rho_3 ])$.
\end{enumerate}

\begin{prp} 
If $\ld_1$ and $\ld_2$ are cohomologous, then $R^{\ld_1}(G)$ and $R^{ \ld_2}(G)$ are isomorphic as ordered abelian group.
\end{prp}
\begin{proof} Since $\ld_1$ and $\ld_2$ are cohomologous, there is some Borel function $\mu\colon G\rightarrow \T$ such that $\ld_2= d\mu\,\ld_1$. Let $\rho \colon G \rightarrow U(H)$ be a $\ld_1$ -representation. Let $L(\rho) \colon G \rightarrow U(H)$ be defined by $L(\rho)(g) = \mu(g)\rho(g)$. Then $L(\rho)$ is a $\ld_2$-representation. It’s an easy exercise to check that $L$ defines an isomorphism of $V^{\ld_1}(G)$ and $V^{\ld_2} (G)$, with inverse defined by $R(\rho)(g)=\mu^{-1}(g)\rho(g)$. Hence it defines an isomorphism of the ordered abelian groups.
\end{proof}
Note however that there is no natural choice of $\mu$ in the above proof, and different choices may lead to different isomorphisms. Now we are ready to define our invariant, which is certain twisted version of equivariant K-theory. The main idea is to generalize the construction of equivariant K-theory by using cocycle representations of the group instead of ordinary representations. In the following we shall fix an action $\af\colon G \cra A$ of a compact group on a C*-algebra $A$. \\
\begin{dfn} (Definition 2.4.1, \cite{Phillips1987}) Let $\ld$ be a 2-cocycle. Let $(G, H_1 , \rho_1 )$ and $(G, H_2 , \rho_2 )$ be two $\ld$-representations. Let $L(H_1 , H_2 )$ be the set of all linear maps from $H_1$ to $H_2$ . If $\dim(H_1 ) = m$ and $\dim(H_2 ) = n$, then we can identify $L(H_1 , H_2 )$ with $M_{m,n} (\C)$. One checks that 
\beq
g \cdot t = \rho_2 (g) \cdot t \cdot \rho_1 (g)^{-1}
\eeq
defines an action of $G$ on $L(H_1 , H_2 )$. \\
We identify $L(H_1 , H_2 ) \otimes A$ with a set of $A$-module homomorphisms from $H_1 \otimes A$ to $H_2 \otimes A$, by formula 
\beq
(t \otimes a)(\xi \otimes b) = t\xi \otimes ab,\quad\quad  t \otimes a \in L(H_1 , H_2 ) \otimes A,\,\, \xi \otimes b \in H_1 \otimes A.
\eeq
When $A$ is unital, we can also identify $L(H_1 , H_2 ) \otimes A$ with $M_{m,n}(A)$.\\
Give $L(H_1 , H_2 ) \otimes A$ the diagonal action, so that $g \cdot (t \otimes a) = (gt \otimes \af_g (a))$, for $t \in L(H_1 , H_2 )$ and $a \in A$. Define the product of $s \in L(H_2 , H_1 ) \otimes A$ and $t \in L(H_3 , H_2 ) \otimes A$ to be their product as $A$-module homomorphisms. Then one has $(s \otimes a)(t \otimes b) = (st \otimes ab)$. We also define an adjoint operation from $L(H_1 , H_2 ) \otimes A$ to $L(H_2 , H_1 ) \otimes A$ by $(t \otimes a) ^\ast = (t^\ast \otimes a^\ast )$. We see that multiplication and adjoint are preserved by the $G$-action.
\end{dfn}
We write $L(H)$ for $L(H, H)$.

\begin{dfn} Let $(G, A, \af)$ be a $G$-C*-algebra. Let $\ld$ be a 2-cocycle. We define $P^\ld (G, A, \af)$ be the the set of pairs $(p, \pi)$, where $(G, H, \pi)$ is a $\ld$-representation and $p$ is an invariant projection of $L(H) \otimes A$ (The action on $L(H)$ is induced by $\pi$). If $(p_1 , \pi_1)$ and $(p_2 , \pi_2)$ are two
such pairs, we define an addition
\beq
(p_1, \pi_1 ) \oplus (p_2, \pi_2)\colon=(\diag\{p_1 , p_2\},  \pi_1\oplus\pi_2)
\eeq
Two pairs $(p_1 , \pi_1)$ and $(p_2 , \pi_2)$ are called Murray-von Neumann equivalent if there is an $G$-invariant element $u \in L(H_1 , H_2 ) \otimes A$ such that $u ^* u = p$ and $uu ^\ast = q$.\\
We let $V^\ld (G, A, \af)$ be the set of equivalence classes in $P^\ld (G, A, \af)$. The addition in $P^{\ld}(G, A, \af)$ respect the equivalence relation, therefore defines an addition on $V^{\ld}(G, A, \af)$. Just like ordinary K-theory, one can show that $V^{\ld}(G, A, \af)$ is an abelian semigroup with identity.\\
If $\phi \colon (G, A, \af) \rightarrow (G, B, \bt)$ is an equivariant homomorphism of $G$-C*-algebras, we define
\beq
\phi^\ast\colon V^\ld (G, A, \af) \rightarrow V^\ld (G, B, \bt)\quad\text{by}\quad\phi^\ast ([(p, \pi)]) = [((\id_{L(H)} \otimes \phi)(p), \pi)].
\eeq
\end{dfn}

\begin{lem}(Lemma 2.4.3, \cite{Phillips1987}) $V^\ld(\ast)$ is a covariant functor from $G$-C*-algebras to abelian semigroups with identity.
\end{lem}

\begin{dfn}
 If $A$ is unital, we define $K_0^\ld (G, A, \af)$ to be the Grothendieck completion of $V^\ld (G, A, \af)$. If $A$ is non-unital, Let $\tilde{A}$ be the unitalization of $A$ and $\tilde{\af}$ be the natural extension of $\af$. Let $p\colon (G, \tilde{A}, \tilde{\af})\rightarrow (G, A, \af)$ be the canonical projection. Then $K_0^\ld (G, A, \af)$ is defined to be $\ker p^\ast$. If $G$ and $A$ are clear from the context, we shall simply write $K_0^{\ld}(\af)$. We shall write $K_0^{\ld}(\af)_+$ for the image of $V^{\ld}(G, A, \af)$ in $K_0^{\ld}(\af)$ under the canonical map.
\end{dfn}
When $A$ is stably finite, it becomes an ordered abelian group with $K_0^{\ld}(\af)_+$ being the positive cone. One checks that $K_0^\ld(\ast)$ is a covariant functor from stably finite $G$-C*-algebras to ordered abelian groups.\\

Now for any pairs of 2-cocycles $\ld_1$ and $\ld_2$ , there is a natural map:
\beq
R^{\ld_1} (G) \times K_0^{\ld_2} (\af) \rightarrow K_0^{\ld_1\ld_2} (\af), 
\eeq
induced by
\beq
[(G, H_1 , \pi_1 )] \times [(p, \pi_2 )] \rightarrow [(p \otimes 1, \pi_2 \otimes \pi_1 )],
\eeq
Which becomes an element in
\beq
Hom(R^{\ld_1} (G),\,\,\, Hom(K_0^{\ld_2}(\af), K_0^{\ld_1 \ld_2} (\af))).
\eeq
We shall call this map the \emph{partial action} of $R^{\ld_1}(G)$ on $K_0^{\ld_2}(\af)$. The reason for using this terminology is that, if we take the direct sum of $K_0^{\ld}(\af)$ for different 2-cocycle $\ld$, then $R^{\ld_1}(G)$ indeed gives an action. As a special case, when $\ld_1$ is trivial, the partial action of $R^{\ld_1}(G)=R(G)$ gives $K_0^{\ld_2} (\af)$ a module structure over the representation ring.\\

One can check that this map is compatible with the pairing between cocycle representation groups in the sense that the following diagram commutes:
\\

\[
\begin{tikzcd}
R^{\ld_1} (G) \times R^{\ld_2} (G) \times K_0^{\ld_3} (\af)\arrow{r}\arrow{d} &R^{\ld_1 \ld_2} (G) \times K_0^{\ld_3} (\af)\arrow{d}\\
R^{\ld_1} (G) \times K_0^{\ld_2 \ld_3} (\af) \arrow{r} &K_0^{\ld_1 \ld_2 \ld_3} (\af)
\end{tikzcd}
\]

When $A$ is unital, there is one special element lives in $K_0(\af)$ ($\ld$ is trivial) which is represented by $(1_A , \pi_t)$, where $\pi_t$ is the trivial representation. We shall denote this element by $[1_\af]$.\\
Let $\phi \colon (G, A, \af) \rightarrow (G, B, \bt)$ be an equivariant homomorphism. Then we can check that the induced map
\beq
\phi ^\ast \colon K_0^\ld (\af) \rightarrow K_0^\ld (\bt)
\eeq
is compatible with the partial actions of the cocycle representation groups in the sense that the following diagram commutes:

\[
\begin{tikzcd}
K_0^{\ld_2} (\af)\arrow{r}{\phi^\ast}\arrow{d}[left]{R^{\ld_1}(G)} & K_0^{\ld_2} (\bt)\arrow{d}[left]{R^{\ld_1}(G)}\\
K_0^{\ld_1\ld_2} (\af)\arrow{r}{\phi^\ast} & K_0^{\ld_1\ld_2} (\bt)
\end{tikzcd}
\]

Just like equivariant K-theory, there is also a module picture of the twisted equivariant K-theory. One can also talk about $K_1$, but we will not need it in this paper.\\

The following theorem could be proved by mimicking the proof of the original Julg's theorem. We refer the reader to \cite{Packer1989} for the definition of twisted crossed product and its basic properties.
\begin{thm}
(Julg's theorem) Let $(G, A, \af)$ be an $G$-C*-algebra, where $G$ is compact. Then:
\beq
K_0^\ld (\af) \cong K_0(A\rtimes_{\af,\bar{\ld}} G),
\eeq
where $\bar{\ld}$ is the complex conjugate of $\ld$ (viewed as a complex function), and $A \rtimes_{\af, \bar{\ld}} G$ is the twisted crossed product.
\end{thm}
The $K^{\ld}$-groups are sequentially continuous and commute with direct sums:
\begin{prp}\label{sc}
 If $(G, A, \af) =\dlim (G, A_n , \af_ n)$, then for any 2-cocycle $\ld$, we have
\beq
K_0^{\ld} (\af) \cong \dlim K_0^{\ld} (\af_n).
\eeq
Moreover, the partial actions of $R^{\ld}(G)$ commutes with inductive limit.
\end{prp}

\begin{prp}\label{ds}
 Let $(G, A, \af)$ and $(G, B, \bt)$ be two $G$-C*-algebras, then for any 2-cocycle $\ld$, we have
\beq
K_0^\ld (G, A \oplus B, \af \oplus \bt) \cong K_0^\ld (\af) \oplus K_0^\ld (\bt),
\eeq
which are compatible with the partial actions of the cocycle representation groups.
\end{prp}

We now give some examples of twisted equivariant K-theory.
\begin{prp}\label{k0ma}
Let $\af\colon G\cra M_n$ be an action induced by a $\ld_1$-representation $\rho$. Then there is an ordered group isomorphism between $K_0^{\ld_2} (\af)$ and $R^{\ld_1 \ld_2} (G)$, for each 2-cocycle $\ld_2$. Furthermore, under these isomorphisms, the partial action of $R^{\ld_3} (G)$ on $K_0^{\ld_2} (\af)$ becomes the pairing between $R^{\ld_3}(G)$ and $R^{\ld_1 \ld_2} (G)$ (induced by tensor product). The special element $[1_\af]$ becomes $[\rho]$.
\end{prp}
\begin{proof} The map between $K_0^{\ld_2} (\af)$ and $R^{\ld_1 \ld_2} (G)$ is induced by:
\beq
[(p, \pi)] \rightarrow [(\pi \otimes \rho)\,\vert\,_p ].
\eeq
It’s routine to check that this is a well-defined group homomorphism which preserves the positive cone. It’s also easy to see that this map is injective. To show that this map is surjective, we need only to show that any $\ld_1\ld_2$-representation $\pi'$ is in the image. Let $\bar{\rho}$ be the contragredient representation of $\rho$, then $\bar{\rho}\otimes \rho$ is an ordinary representation. It contains the trivial representation $\pi_t$ as a subrepresentation, because
\beq 
\langle \chi_{\bar{\rho}\otimes \rho} , \chi_{\pi_t}\rangle = \int_G \chi_{\bar{\rho}}(g)\chi_{\rho}(g)\, dg = \langle \chi_\rho, \chi_\rho\rangle > 0.
\eeq

Let $e$ be the 1-dimensional projection correspond to a trivial subrepresentation of $\bar{\rho}\otimes \rho$. Now consider $[(1\otimes e, \pi'\otimes \bar{\rho})]$, which is evidently an element of $K_0^{\ld_2} (\af)$ being mapped to $[\pi']$. \\
Finally, it following from the definition of the isomorphism that $[1_\af]$ corresponds to $[\rho]$.
\end{proof}

\begin{exa}\label{z3z3}
Let $G=\Z/3\Z\oplus \Z/3\Z$. Let $\pi_i$, $\ld_i$ and $\af_i$ be constructed as in Example \ref{fdce}. Then $R^{\ld_i}(G)\cong \Z$, generated by $[\pi_i]$.
\end{exa}
\begin{proof}
See Lemma 2.4 of \cite{Higgs2001}.
\end{proof}

\section{Equivariant Classification}
\begin{dfn}
Let $\af\colon G\cra A$ be an action, where $G$ is compact. The \emph{scale} of $K_0(\af)$ is
\beq
D(\af)=\{[(p, \pi_t)]\in K_0(\af)\,\vert\, p\in A, \pi_t \text{\,\,is the trivial representation}\}.
\eeq
We use $\xi_\af$ to denote $[1_\af]$ if $A$ is unital, and the scale $D(\af)$ if $A$ is non-unital. 
\end{dfn}
Following the convention in the Elliott's classification program, we will write the invariant of $(G, A, \af)$ by $Ell(G, A, \af)$, or $Ell(\af)$ for short if $G$ and $A$ is clear from the context. It consists of the collection of the ordered abelian groups ${K_0^\ld(\af)}$ together with $\xi_\af$, and all the partial actions by $R^\ld(G)$.\\
Let $\af\colon G\cra A$ and $\bt\colon G\cra B$ be two actions. A homomorphism
\beq
T \colon Ell(\af) \rightarrow Ell(\bt)
\eeq 
is understood to be a collection of ordered group homomorphisms
\beq
T^\ld \colon (K_0^\ld (\af), K_0^\ld(\af)_+) \rightarrow (K_0^\ld(\bt), K_0^\ld(\bt)_+),
\eeq
which are compatible with the partial actions of the cocycle representations groups. That is, the following diagram is commutative for any 2-cocycles $\ld_1$ and $\ld_2$:
\[
\begin{tikzcd}
K_0^{\ld_2} (\af)\arrow{r}{T^{\ld_2}}\arrow{d}[left]{R^{\ld_1}(G)} & K_0^{\ld_2} (\bt)\arrow{d}[left]{R^{\ld_1}(G)}\\
K_0^{\ld_1\ld_2} (\af)\arrow{r}{T^{\ld_1}} & K_0^{\ld_1\ld_2} (\bt)
\end{tikzcd}
\]
We say $T$ is \emph{unital} if $T$ sends $[1_\af]$ to $[1_\bt]$. We say $T$ is \emph{contractive} if it preserves the scales. We say $T$ is an \emph{isomorphism} if each $T^{\ld}$ is an isomorphism and it sends $\xi_\af$ to $\xi_\bt$.\\
Now we are ready to state the main theorem of this paper:
\begin{thm}
 Let $G$ be a compact group. Let $A$, $B$ be two AF algebras. Let $\af \colon G \cra A$ and $\bt \colon G \cra B$ be two inductive limit actions which are locally spectrally trivial. Suppose $A$, $B$ are either both unital or both non-unital. Then $\af$ and $\bt$ are conjugate if and only if
\beq
Ell(\af) \cong Ell(\bt).
\eeq
\end{thm}

Just like classification of AF algebras, the proof is divided into two steps. The first step is to establish an existence theorem and a uniqueness theorem for actions on finite dimensional C*-algebras. The second step is to use the intertwining argument to get the desired isomorphisms. Since we are assuming that the actions on finite dimensional C*-algebras are spectrally trivial, they are direct sum of actions on each simple summand. So let's look at what happens for actions on matrix algebras first.\\
\begin{prp}
 Let $\af\colon G\cra M_n$ and $\bt \colon G\cra M_k$ be two actions. Suppose that they are induced by some cocycle representations $\pi_\af$ and $\pi_\bt$, respectively. Then there is a nontrivial equivariant homomorphism between the two $G$-C*-algebras if and only if there is some cocycle representation $\pi_\gm$ and some nonzero projection $p \in M_k$, invariant under $\pi_\bt$, such that $\pi_\bt\,\vert\,_p$ is equivalent to $\pi_\af \otimes \pi_\gm$ .
\end{prp}
\begin{proof} One direction is easy: if such cocycle representation exist, then the map $a \rightarrow U p(a\otimes1)pU ^\ast$ is equivariant, where $U$ is a unitary implementing the equivalence between $\pi_\bt\,\vert\,p$ and $\pi_\af \otimes\pi_\gm$. For the other direction, let $T$ be a nontrivial equivariant homomorphism. Set $p = T(1)$, which is a nonzero projection in $M_k$. It's easy to check that $p$ is invariant under $\bt$. Restricting the action to $\bt\,\vert\, pM_kp$, we may without loss of generality assume that $T$ is unital.\\
Let $A = T (M_n )^{\prime}$ , the commutant of $T(M_n)$. Using the fact that $T$ is equivariant, we can show that $A$ is a subalgebra of $M_k$ invariant under $\bt$. Consider the map $S \colon A \otimes M_n \rightarrow M_k$ , induced by $S(a \otimes b) = aT(b)$. This map is an equivariant homomorphism if we define the action on $A\otimes M_n$ by $(\bt\,\vert\,_A) \otimes \af$. Now it's elementary to check that $A$ is actually a matrix algebra, say isomorphic to $M_r$, such that $M_r \otimes M_n \cong M_k$ . (In particular, $S$ is an isomorphism). Let $\pi_\gm$ be a cocycle representation on $M_r$ inducing $\bt\,\vert\,_A$ under the isomorphism of $A$ and $M_r$. Let $\gm$ be the action on $M_r$ induced by $\pi_\gm$ . Then up to conjugate, $\bt$ and $\af \otimes \gm$ are the same. Hence
\beq
\pi_\bt = \mu_g U (\pi_\af \otimes \pi_\gm )U ^\ast ,\quad \text{for some\,\,} U \in M_k \text{\,\,and\,\,} \mu_g \in \T.
\eeq
Replacing $\pi_\gm$ by $\mu\pi_\gm$, we get the desired cocycle representation.
\end{proof}
\begin{prp}\label{homma}
 Let $\af \colon G \cra M_n$ and $\bt \colon G \cra M_k$ be two actions. Then for any unital homomorphism $\Gm \colon Ell(\af) \rightarrow Ell(\bt)$, there is an unital equivariant homomorphism $T \colon M_n \rightarrow M_k$ such that $T^\ast = \Gm$. Furthermore, if $T_1$ and $T_2$ are two unital equivariant homomorphism such that $T_1^\ast = T_2^\ast$ , then $T_1$ and $T_2$ are conjugate by some unitary.
\end{prp}
\begin{proof}
Suppose $\af$ is induced by a $\ld_\af$-representation $\pi_\af$, $\bt$ is induced by a $\ld_\bt$ -representation $\pi_\bt$. By Proposition \ref{k0ma}, there is an ordered group isomorphism between $K_0^\ld(\af)$ and $R^{\ld\ld_\af} (G)$, for each 2-cocycle $\ld$. Furthermore, under these isomorphisms, the partial action of $R_0^\ld(G)$ on $K_0^\ld(\af)$ becomes the pairing between $R^\ld(G)$ and $R^{\ld\ld_\af} (G)$ (induced by tensor product). The special element $[1_\af]$ becomes $[\pi_\af]$. Similar conclusion holds for $\bt$. We now claim that, under these isomorphisms, there is some $\overline{\ld_\af}\ld_\bt$-representation $\pi_\gm$, such that the homomorphisms $\Gm$ becomes
\beq
\Gm^\ld \colon R^{\ld\ld_\af} (G) \rightarrow R^{\ld\ld_\bt} (G), \quad\Gm^\ld ([\pi]) = [\pi \otimes \pi_\gm ]
\eeq
To see this, consider the map 
\beq
\Gm^{\overline{\ld_\af}} \colon R^{\overline{\ld_\af}\ld_\af} (G)=R(G) \rightarrow R^{\overline{\ld_\af} \ld_\bt} (G).
\eeq
Let $\pi_t$ be the trivial representation. Then $[\pi_t] \in R(G)$. Choose $\pi_\gm$ such that $[\pi_\gm] = \Gm^{\overline{\ld_\af}} [\pi_t]$. Now let $[\pi] \in R^{\ld\ld_\af} (G)$ be an arbitrary element. Since the homomorphisms $\Gm^\ld$ is compatible with the partial actions, we have
\beq
\Gm^\ld ([\pi]) =[\pi]\Gm^{\ld} ([\pi_t ]) = [\pi][\pi_\gm] = [\pi \otimes \pi_\gm ],
\eeq
which completes the proof of our claim.\\
Since $[\pi_\af \otimes \pi_\gm ] = [\pi_\bt ]$, we see that $\pi_\af \otimes \pi_\gm$ and $\pi_\bt$ are equivalent. Now let $\gm \colon G \cra M_r$ be the action induced by $\pi_\gm$ . We have $M_n\otimes M_r\cong M_k$. The map 
\beq
M_n \rightarrow M_n \otimes M_r,\quad a\rightarrow a \otimes 1
\eeq
induces an equivariant map $T \colon M_n \rightarrow M_k$. From the construction of $T$ we see that $T^\ast$ is induced by tensor with $[\pi_{\gm}]$,  hence $T ^\ast = \Gm$.\\
Now suppose $T_1$ and $T_2$ are two equivariant maps such that $T_1 ^\ast = T_2 ^\ast$ . From the above proof, we may assume that $T_i$ is induced by tensoring with some cocycle representations $\pi_{\gm_i}$ , for $i = 1, 2$. Since $T_1 ^\ast ([1_{\af}]) = T_2 ^\ast ([1_\af])$, we have that $\pi_\af \otimes \pi_{\gm_1}$ and $\pi_\af \otimes \pi_{\gm_2}$ are equivalent. Therefore $T_1$ and $T_2$ are conjugate by some unitary.
\end{proof}

\begin{rmk}\label{rmkhomma}
From the proof of the above proposition, we can actually see that for any homomorphism $\Gm\colon Ell(G, M_n, \af)\rightarrow Ell(G, B, \bt)$, the homomorphism is entirely determined by the image of $[\pi_t]$, where $[\pi_t]$ is the equivalent class of the trivial representation, viewed as an element of $K^{\overline{\ld_\af}}(G, M_n, \af)$ under the isomorphism $K^{\overline{\ld_\af}}(G, M_n, \af)\cong R(G)$.
\end{rmk}
One can also get a non-unital version from the above proof.
\begin{cor}\label{nuhomma}
Let $\af \colon G \cra M_n$ and $\bt \colon G\cra M_k$ be two actions. Let
\beq
\Gm \colon Ell(\af) \rightarrow Ell(\bt)
\eeq
be a contractive homomorphism. Then there is an equivariant homomorphism $T \colon M_n \rightarrow M_k$ such that $T ^\ast = \Gm$. Furthermore, if $T_1$ and $T_2$ are two equivariant homomorphisms such that $T_1 ^\ast = T_2 ^\ast$ , then $T_1$ and $T_2$ are conjugate by some unitary.
\end{cor}
We can now extend Proposition \ref{homma} to actions on finite dimensional C*-algebras which are spectrally trivial:
\begin{prp}\label{homfd}
 Let $A$, $B$ be two finite dimensional C*-algebras, let $\af \colon G\cra A$ and $\bt \colon G \cra B$ be two spectrally trivial actions. Then for any contractive homomorphism $\Gm \colon Ell(\af) \rightarrow Ell(\bt)$, there is an equivariant homomorphism $T \colon A \rightarrow B$ such that $T ^\ast = \Gm$. We can choose $T$ to be unital if $\Gm$ is unital. Furthermore, if $T_1$ and $T_2$ are two equivariant homomorphism such that $T_1 ^\ast = T_2 ^\ast$ , then $T_1$ and $T_2$ are conjugate by some unitary.
\end{prp}
\begin{proof}
A finite dimensional C*-algebra is a direct sum of matrix algebras, and a spectrally trivial action is a direct sum of actions on each matrix algebra. Let $A = A_1 \oplus A_2 \oplus \cdots  A_l$ , let $B = B_1 \oplus B_2 \oplus \cdots B_s$ , where $A_i = M_{n_i}$ and $B_j = M_{k_j}$ are matrix algebras. Let $\af_i$ be the induced action on $A_i$ and $\bt_j$ be the induced action on $B_j$ . In view of Proposition \ref{ds}, we may write
\beq
Ell(\af) \cong \oplus_{i=1}^l Ell(G, A_i, \af_i ),\quad Ell(\bt) \cong \oplus_{j=1}^s Ell(G, B_j , \bt_j). 
\eeq
Now let $\Gm_{i,j}$ be the partial maps from $Ell(G, A_i , \af_i)$ to $Ell(G, B_j, \bt_j)$ induced by $\Gm$. Then by Corollary \ref{nuhomma}, there exists equivariant homomorphisms $\phi_{i,j} \colon A_i
 \rightarrow B_j$ such that $\phi ^\ast_{ i,j} = \Gm_{i,j}$. Since $T$ is contractive, we have
\beq
\oplus_{i=1}^l \Gm_{i, j} ([1_{\af_i}]) = [(p, \pi_t)],
\eeq
 where $\pi_t$ is the trivial representation and $p\in B_j$ is invariant under $\bt_j$. Under the identification of $K^{\ld}(G, B_j, \bt_j)$ and $R^{\ld\ld_{\bt_j}}(G)$, we can see that $\Gm_{i, j}([1_{\af_i}])$ are equivalent to subrepresentations of $\pi_{\bt_j}\,\vert\,_p$ such that the direct sum is equal to $\pi_{\bt_j}\,\vert\,_p$. Conjugating suitable unitaries if necessary, we can arrange that the homomorphisms $\phi_{i,j}$ have orthogonal range. Then the partial maps $\phi_{i,j}$ will give a homomorphism $\phi \colon A \rightarrow B$ such that $\phi^\ast = \Gm$. It is unital if $\Gm$ is. Arguing as in Proposition \ref{homma}, one can see that if $T_1$ and $T_2$ are two equivariant homomorphism such that $T_1 ^\ast = T_2 ^\ast$ , then $T_1$ and $T_2$ are conjugate by some unitary.
\end{proof}
We can now use the intertwining argument to prove our main result Theorem.

\begin{proof} (Proof of Theorem 4.1) Let $G$ be a compact group. Suppose that 
\beq
(G, A, \af) = \dlim (G, A_n , \af_n)\quad (G, B, \bt) =\dlim (G, B_n , \bt_n),
\eeq
where $A_n, B_n$ are finite dimensional C*-algebras and the actions $\af_n$ and $\bt_n$ are spectrally trivial. Let $\Gm \colon Ell(\af) \rightarrow Ell(\bt)$ be an isomorphism. By
Proposition \ref{sc}, we have
\beq
Ell(\af) \cong \dlim Ell(\af_n),\quad Ell(\bt) \cong \dlim Ell(\bt_n).
\eeq
Set $n_0=1$. Consider the homomorphism $\Gm_{n_0} \colon Ell(\af_{n_0}) \rightarrow Ell(\bt)$ which is the composition of the connecting map and $\Gm$. We want to show that this homomorphism pulls back to finite stage. Since $(G, A_{n_0}, \af_{n_0})$ is a finite direct sum of actions on matrix algebras, without loss of generality we may assume that $A_{n_0}$ is a matrix algebra. Suppose $\af_{n_0}$ is induced by some $\ld_{n_0}$-representation $\pi_{n_0}$. From Remark \ref{rmkhomma}, we see that $\Gm_{n_0}$ is determined by the image of $[\pi_t]$. Here $[\pi_t]$ is the equivalent class of the trivial representation, viewed as an element of $K^{\overline{\ld_{n_0}}}(\af)$ under the isomorphism
\beq
K^{\overline{\ld_{n_0}}}(\af)\cong R^{\overline{\ld_{n_0}}\ld_{n_0}}(G) =R(G).
\eeq
 The image of $[\pi_t]$ can certainly be pulled back to finite stage: there is some $m_0$ and certain element $c \in K_0^{\overline{\ld_n}} (G, B_{m_0} , \bt_{m_0})$ such that the image of $c$ in $K_0^\ld(\bt)$ is $\Gm^{\overline{\ld_{n_0}}} ([\pi_t])$.\\
Now for any 2-cocycle $\ld$, we can define a homomorphism
\beq
\Theta^\ld \colon K_0^\ld(G, A_{n_0} , \af_{n_0}) \rightarrow K_0^\ld(G, B_{m_0} , \bt_{m_0}),\quad \theta([\pi]) = [\pi]c.
\eeq
It's then easy to verify that $\{\Theta^{\ld}\}$ defines homomorphisms which are compatible with the partial actions by $R^\ld(G)$.\\

If $A$, $B$ are unital, enlarge $m_0$ if necessary, we may assume that $\Theta([1_{\af_{n_0}} ]) = [1_{\bt_{m_0}}]$. If $A$, $B$ are both non-unital, since $\Gm_{n_0}[1_{\af_{n_0}}]\in D(\bt)$, enlarge $m_0$ if necessary, we may assume that $\Theta([1_{\af_{n_0}}])=[(p, \pi_t)]$, for some $p\in B_{m_0}$ ($\pi_t$ is the trvial representation). In any case, we have the following commutative diagram.
\[
\begin{tikzcd}
Ell(\af_ {n_0})\arrow[dashed]{d}{\Theta} \arrow{rd}{\Gm_{n_0}}&\\
Ell(\bt_{m_0})\arrow{r} &Ell(\bt)
\end{tikzcd}
\]
While the homomorphisms in the above diagram are all contractive. Similarly, the map $\Gm^{-1}_{m_0} \colon Ell(\bt_{m_0} ) \rightarrow Ell(\af)$ can be pulled back to some $Ell(\af_{n_1})$. By Remark \ref{rmkhomma}, if $\gm \colon G \cra M_n$ is an action on matrix algebra and $\theta \colon G\cra C$ is an arbitrary action, then any homomorphism (not necessarily preserve the special element)
\beq
Ell(G, M_n , \gm) \rightarrow Ell(G, C, \theta)
\eeq
is determined by the image of the trivial representation. Enlarge $n_1$ if necessary, we can match the images of the trivial representations in the following diagram, hence making it commutative:
\[
\begin{tikzcd}
Ell(\af_ {n_0})\arrow{r}\arrow[dashed]{rd}& Ell(\af_{n_1} )\arrow{r} & Ell(\af_{n_1} ) \\
 &Ell(\bt_{m_0})\arrow[dashed]{u}\arrow{ur}{\Gm^{-1}_{m_0}}&
\end{tikzcd}
\]
We can further assume that all homomorphisms are contractive by enlarging $n_1$ if necessary.\\
Continuing this way, we get the following commutative diagram:
\[
\begin{tikzcd}
Ell(\af_{n_0})\arrow{r} \arrow[dashed]{d}& Ell(\af_{n_1}) \arrow{r}\arrow[dashed]{d}\cdots\arrow{r}& Ell(\af)\arrow{d}\\
Ell(\bt_{m_0})\arrow{r} \arrow[dashed]{ur}& Ell(\af_{m_1}) \arrow{r}\arrow[dashed]{ur}\cdots\arrow{r}& Ell(\bt)\arrow{u}\\
\end{tikzcd}
\]
By Proposition \ref{homfd} we can lift the commutative diagram to the diagram
\[
\begin{tikzcd}
(G, A_{n_0} , \af_{n_0})\arrow{r} \arrow{d}& (G, A_{n_1} , \af_{n_1}) \arrow{r}\arrow{d}\cdots\arrow{r}& (G, A, \af)\arrow[dashed]{d}\\
(G, B_{m_0} , \bt_{m_0})\arrow{r} \arrow{ur}&(G, B_{m_1} , \af_{m_1}) \arrow{r}\arrow{ur}\cdots\arrow{r}&(G, B, \bt)\arrow[dashed]{u}\\
\end{tikzcd}
\]
By Proposition \ref{homfd} again, the above diagram commutes up to conjugation by unitaries. Hence we can make it truly commutative by conjugating suitable unitaries. Let $T \colon (G, A, \af) \rightarrow (G, B, \bt)$ be the induced homomorphism. Then it's easy to see that $T$ is an equivariant isomorphism such that $T^\ast = \Gm$.
\end{proof}

\section{Simplification of the invariant}
Let $\af\colon G\cra A$ be an action, where $G$ is compact. When we define the invariant $Ell(\af)$, we made use of all possible 2-cocycles $\ld$. There is some redundancy that we can get rid of. First, we note the following:
\begin{prp}
Let $\ld$ and $\ld_1$ be two cohomologous 2-cocycle. Then there is an isomorphism between $K^{\ld}(\af)$ and $K^{\ld_1}(\af)$.
\end{prp}

\begin{proof}
Let $\mu\colon G\rightarrow \T$ be a Borel map such that $\ld_1=(d\,\mu)\ld$. Let $[(p, \pi)]$ be an element in $V^{\ld}(\af)$. Let $\pi_1=\mu\pi$. It's easy to check that $\pi_1$ defines a $\ld_1$-representation. Since $\pi_1$ and $\pi$ induce the same action, we see that $p$ is invariant under $\af\otimes \Ad \pi_1$. Hence $[(p, \pi_1)]$ defines an element in $V^{\ld_1}(\af)$. Let us denote this map $[(p, \pi)]\rightarrow [(p, \pi_1)]$ by $f_\mu$. One can check from the definition that $f_\mu$ is a semigroup homomorphism, with the obvious inverse map $f_{\mu^{-1}}$. Hence $f_{\mu}$ defines an semigroup isomorphism, which induces an isomorphism between $K^{\ld}(\af)$ and $K^{\ld_1}(\af)$.
\end{proof}

Now we can show that, in order to have a homomorphism for the invariants, it's enough to use a complete set of representatives of 2-cocycles containing the trivial 2-cocycle.
\begin{prp}\label{isoK}
Let $(G, A, \af)$ and $(G, B, \bt)$ be two $G$-C*-algebras, where $G$ is compact. Let $\{\ld_i \,\vert\, i\in I\}$ be a complete set of representatives of $H_m^2(G, \T)$. Suppose we have a collection of homomorphisms
\beq
\Gm^{\ld_i} \colon (K_0^{\ld_i} (\af), K_0^{\ld_i}(\af)_+) \rightarrow (K_0^{\ld_i}(\bt), K_0^{\ld_i}(\bt)_+),
\eeq
which are compatible with the partial actions of all cocycle representation groups of the form $R^{\ld_i^{-1}\ld_j}(G)$:
\[
\begin{tikzcd}
K_0^{\ld_i} (\af)\arrow{r}{\Gm^{\ld_i}}\arrow{d}[left]{R^{\ld_i^{-1}\ld_j}(G)} & K_0^{\ld_i} (\bt)\arrow{d}[left]{R^{\ld_i^{-1}\ld_j}(G)}\\
K_0^{\ld_j} (\af)\arrow{r}{\Gm^{\ld_j}} & K_0^{\ld_j} (\bt)
\end{tikzcd}
\] 
Then there is a homomorphism $T\colon Ell(\af)\rightarrow Ell(\bt)$ such that $T^{\ld_i}=\Gm^{\ld_i}$, for each $i\in I$.
\end{prp}
\begin{proof}
Any two 2-cocycle $\ld$ is cohomologous to some $\ld_i$. Pick any Borel function $\mu$ such that $\ld=(d\,\mu)\ld_i$. Let the map $f_\mu$ be defined as in the proof of \ref{isoK}. We define a map:
\beq
T^{\ld}\colon (K_0^{\ld} (\af), K_0^{\ld}(\af)_+) \rightarrow (K_0^{\ld}(\bt), K_0^{\ld}(\bt)_+)
\eeq
by $T^{\ld}=f_\mu\circ \Gm^{\ld_i}\circ f_{\mu^{-1}}$. We want to show that these homomorphisms are compatible with the partial actions. So let $\om_i$, $\om_j$ be two 2-cocycles such that $\om_i=(d\,\mu_i) \ld_i$ and $\om_j=(d\,\mu_j) \ld_j$. Let $T^{\om_i}, T^{\om_j}$ be defined as above. Let $\pi$ be a $\om_i^{-1}\om_j$-representation. Let $\tilde{\pi}=\mu_i\mu_j^{-1}\pi$, which is a $\ld_i^{-1}\ld_j$-representation. We then have the following diagram:

\[
\begin{tikzcd}
& K_0^{\ld_i}(\af)\arrow{rr}{\Gm^{\ld_i}\quad\quad\quad} \arrow{dd}[near start]{\lbrack\tilde{\pi}\rbrack} & & K_0^{\ld_i}(\bt)\arrow{dd}[near start]{\lbrack\tilde{\pi}\rbrack} \arrow{dl}{f_{\mu_i}}\\
K_0^{\om_i}(\af)\arrow{ur}{f_{\mu_i^{-1}}} \arrow{dd}[near start]{\lbrack\pi\rbrack} \arrow[dashed, crossing over]{rr}{\quad\quad T^{\om_i}}& &K_0^{\om_i}(\bt)\arrow[crossing over]{dd}[near start]{\lbrack\pi\rbrack}\\
& K_0^{\ld_j}(\af)\arrow{rr}{\Gm^{\ld_j}\quad\quad\quad} & & K_0^{\ld_j}(\bt)\arrow{dl}{f_{\mu_j}}\\
K_0^{\om_j}(\af)\arrow{ur}{f_{\mu_j^{-1}}}\arrow[dashed]{rr}{\quad\quad T^{\om_i}}&& K_0^{\om_j}(\bt)\\
\end{tikzcd}
\]

The top and bottom squares are commutative by the definition of $T^{\ld}$, the back square is commutative by assumption. We now check that the right square is commutative. Let $[p, \pi_0]\in K_0^{\ld_i}(\bt)$. Then we have:
\beq
[\pi]f_{\mu_i}([(p, \pi_0)])=[\pi][(p, \mu_i\pi_0)]=[p\otimes 1, (\mu_i\pi_0)\otimes \pi].
\eeq
Therefore
\begin{align*}
&f_{\mu_j}([\tilde{\pi}][(p, \pi_0)])=f_{\mu_j}([(p\otimes 1, \pi_0\otimes \tilde{\pi})])\\
&=[(p\otimes 1, \mu_j(\pi_0\otimes \mu_i\mu_j^{-1}\pi))]\\
&=[(p\otimes 1, (\mu_i\pi_0)\otimes \pi)]=[\pi]f_{\mu_i}([(p, \pi_0)]).\\
\end{align*}
Hence the right square is commutative. Similar computation shows that the left square is commutative. Hence the front square is also commutative, which finishes the the proof.
\end{proof}


\bibliographystyle{amsplain}

\begin{thebibliography}{10}
\bibitem{Bratteli1991}
O.~Bratteli, George~A Elliott, D~E Evans, and A.~Kishimoto, \emph{On the
  classification of inductive limits of inner actions of a compact group},
  Current topics in operator algebras (1991), 13--24.

\bibitem{Dean2009}
Andrew~J Dean, \emph{Inductive limits of inner actions on approximate interval
  algebras generated by elements with finite spectrum}, J. Ramanujan Math. Soc.
  \textbf{24} (2009), no.~4, 323.

\bibitem{Elliott1996b}
George~A. Elliott and Hongbing Su, \emph{K-theoretic classification for
  inductive limit {$\mathbb{Z}_2$} actions on {AF} algebras}, Journal canadien
  de math\'ematiques \textbf{48} (1996), no.~5, 946--958.

\bibitem{Handelman1985}
David Handelman and Wulf Rossmann, \emph{Actions of compact groups on {AF}
  {$C^\ast$}-algebras}, Illinois J. Math. \textbf{29} (1985), no.~1, 51--95.

\bibitem{Higgs2001}
Russell~J. Higgs, \emph{Projective representations of abelian groups}, J.
  Algebra \textbf{242} (2001), 769--781.

\bibitem{Karpilovsky1985projective}
Gregory Karpilovsky, \emph{Projective representations of finite groups},
  vol.~94, Marcel Dekker Inc, 1985.

\bibitem{Kishimoto1990}
A.~Kishimoto, \emph{Actions of finite groups on certain inductive limit
  {C}*-algebras}, Internat. J. Math. \textbf{01} (1990), no.~03, 267--292.

\bibitem{Li2013}
Zhiqiang Li and Wenda Zhang, \emph{A {K}-theoretic classification for certain
  {$\mathbb{Z}/p\mathbb{Z}$} actions on {AF} algebras}, ar{X}iv preprint
  ar{X}iv:1304.0813 (2013).

\bibitem{Moore1976}
Calvin~C Moore, \emph{Group extensions and cohomology for locally compact
  groups. {III}}, Trans. Amer. Math. Soc. \textbf{221} (1976), no.~1, 1--33.

\bibitem{Packer1989}
Judith~A Packer and Iain Raeburn, \emph{Twisted crossed products of
  {C}*-algebras}, Math. Proc. Cambridge, vol. 106, Cambridge Univ Press, 1989,
  pp.~293--311.

\bibitem{Phillips1987}
N.~C. Phillips, \emph{Equivariant {K}-theory and freeness of group actions on
  {C}*-algebras}, Springer-Verlag, 1987.

\bibitem{Phillips2011}
N~Christopher Phillips, \emph{The tracial {R}okhlin property for actions of
  finite groups on {C}*-algebras}, Amer. J. Math. \textbf{133} (2011), no.~3,
  581--636.

\bibitem{Su1996}
Hongbing Su, \emph{K-theoretic classification for certain inductive limit
  {$\mathbb{Z}_2$} actions on real rank zero {C}*-algebras}, Trans. Amer. Math.
  Soc. \textbf{348} (1996), no.~10, 4199--4230.

\bibitem{Wang2015}
Qingyun Wang, \emph{Characterization of product-type actions with the {R}okhlin
  properties}, Indiana Univ. Math. J. \textbf{64} (2015), 295--308.

\end{thebibliography}

Qingyun Wang, \textsc{Department of Mathematics, University of Toronto}\par\nopagebreak
  \textit{E-mail address}: \texttt{wangqy@math.toronto.edu}

\end{document}